\documentclass[letterpaper, 10 pt, conference]{ieeeconf}  
\IEEEoverridecommandlockouts          
\pdfoutput=1

\usepackage{amsmath,amsfonts,amssymb}
\usepackage{graphics} 
\usepackage{epsfig,color} 
\usepackage{times} 
\usepackage{amssymb}  
\usepackage{titlesec}
\titlespacing{\subsection}{0pt}{\parskip}{-\parskip}
\titlespacing{\subsubsection}{0pt}{\parskip}{-\parskip}
\usepackage{amsthm}
\usepackage{graphicx}
\graphicspath{ {./images/} }
\usepackage{graphicx}
\usepackage{booktabs} 
\usepackage{subfig}
\usepackage[ruled,vlined]{algorithm2e}

\newtheorem{lem}{Lemma}
\newtheorem{thm}{Theorem}

\newtheorem{assump}{Assumption}
\newtheorem{remark}{Remark}

\usepackage[bookmarks=true,pageanchor,colorlinks,linkcolor=red,anchorcolor=blue, citecolor=blue,urlcolor=blue,hyperfootnotes=false]{hyperref}

\title{\LARGE \bf
Convergence Rates of Distributed Consensus over Cluster Networks:\\ A Two-Time-Scale Approach
}

\author{Amit Dutta,\quad Almuatazbellah M. Boker,\quad  Thinh T. Doan{$^\star$}
\thanks{{$^\star$}The authors are with the Bradley Department of Electrical and Computer Engineering, Virginia Tech, USA {\tt\small E-mails: \texttt{\{amitdutta,boker,thinhdoan\}@vt.edu}.}}%
}

\begin{document}

\maketitle
\thispagestyle{empty}
\pagestyle{empty}

\begin{abstract}
We study the popular distributed consensus method over networks composed of a number of densely connected clusters with a sparse connection between them. In these cluster networks, the method often constitutes two-time-scale dynamics, where the internal nodes within each cluster reach consensus quickly relative to the aggregate nodes across clusters. Our main contribution is to provide the rate of convergence of the distributed consensus method in a way that characterizes explicitly the impacts of the internal and external graphs on its performance. Our main result shows that the consensus method converges exponentially and only scales with a few number of nodes, which is relatively small to the size of the network.   
The key technique in our analysis is to consider a Lyapunov function that captures the impact of different time-scale dynamics on the convergence of the method. Our approach avoids using model reduction, which is the typical way according to singular perturbation theory and relies on relatively simple definitions of the slow and fast variables. We illustrate our theoretical results by a number of numerical simulations over different cluster networks. 
\end{abstract}

\section{INTRODUCTION}
Cluster networks are omnipresent nowadays. They can be seen, for example, in small-world networks \cite{watts1998collective}, wireless sensor networks \cite{zhao2004wireless} and power systems networks \cite{chow2013power}. Clustering in this context refers to the network topology when there is dense communication inside the clusters and sparse communication among the different clusters. It was shown in \cite{chow1985time} that for these types of networks, the dynamics evolve according to two-time scales. More specifically, the dynamics of the local interactions of the nodes within clusters evolve at faster time-scale relative to the aggregate dynamics of the clusters. This motivates exploiting this intrinsic feature of time scale separation to analyse the performance of these types of networks. This paper is a step in this direction, where our focus is to characterize the convergence rates of the popular distributed control methods over cluster networks, namely distributed consensus averaging. 

Most of the results that exploit the two time-scale phenomenon have been focused on model-reduction, for time-invariant linear dynamics in \cite{chow1985time}, for time-varying linear dynamics in \cite{martin2016time} and for nonlinear dynamics in \cite{biyik2008area}, and control design; see for example \cite{boker}, \cite{mukherjee2021reduced}, \cite{boker2016aggregate,Mor2016,Pham2020c} and the references therein. The two-time scale observation was also leveraged to solve different networked control systems problems. For example, the consensus control problem was solved in \cite{boker}. The leader-follower problem was solved in \cite{ThiemDN2020} in the context of discrete-time systems.

It is known that for (undirected) connected networks, all the nodes will reach consensus regardless of the network topology \cite{mesbahi2010graph}. Furthermore, it is also known that the rate of change of convergence of all nodes to the consensus equilibrium is dependent on how well the network is connected. That is, the more connected the network, with more nodes communicating among each other, the faster the network reaches consensus. However, little attention has been devoted to investigating the inter-dependency of the convergence rates of the different dynamics when the network has a cluster structure. Depending on the application scenario, this can prove to be useful. For example, knowing how each cluster converges to the consensus equilibrium motivates local network design considerations. This, in turn, can lead to distributed feedback strategies to achieve local and global needs simultaneously. Other applications can include distributed learning which involves designing a local learning framework for each cluster which can be shared across the network periodically, thus learning different parts of a global problem at the same time.\\
\textbf{Main contributions}. In summary, our main contribution is to provide the convergence rate of the distributed consensus method over cluster networks, which is missing in the literature. Our main result shows that this rate converges exponentially and only scales with a few number of nodes, which is relatively small to the size of the network. The key technique in our analysis is to consider a Lyapunov function, which is fundamentally different from the one based on model reduction in the existing literature. Finally, we illustrate our theoretical results by a number of numerical simulations over different cluster networks.

\textbf{Technical Approach}. Inspired by singular perturbation theory \cite{kalil2} and recent analysis for the centralized two-time-scale stochastic approximation \cite{Doan2021_SIAM_TwoTimeScaleSA,Doan_two_time_SA2020}, we analyze the convergence properties of cluster consensus networks in a way that respects time-scale separation. More specifically, we show the convergence rates of both the slow and fast dynamics through the use of a composite Lyapunov function. This function is scaled by a singular-perturbation parameter that represents the clustering strength and simultaneously captures the time separation. This approach is different than the typical singular perturbation approach reported in \cite{kalil2} and \cite{biyik2008area} in that it is not based on reducing the system model into two smaller models. This means that our approach does not require the system to be modeled in the standard singular perturbation form, which requires the fast subsystem dynamics to have distinct (isolated) real roots.

\textbf{Novelty}. Existing literature for general network shows that the convergence time depends on the number of nodes, which is much larger than the number of clusters. Hence our analysis provides more compact result than the existing works.  In particular, our analysis highlights how the nodes interact within their cluster versus their external communication with other nodes in other clusters. Knowing how the dynamics within and across clusters evolve provides an useful approach to design better distributed control strategies in different applications including robotics and power networks.







 \section{Problem Formulation} 
We consider an averaging-consensus problem over a network of $N$ nodes with $M$ number of links where the states of the nodes $i = 1,2,..,N$ are represented by $x_{i} \in \mathbb{R}$. In this problem, each node $i$ initializes with a constant $c_{i}$ and their goal is to cooperatively compute the average $\bar{c}$ of their constants, i.e., $\bar{c} = \frac{1}{N}\sum_{i=1}^{N}c_{i}.\label{prob:c_bar}$.

We assume that the nodes can communicate with each other, where the communication structure is described by an undirected and connected graph $G = (V,E)$, where $V = \{1,\ldots,N\}$ and $E = V\times V$ are the vertex and edge sets, respectively. Let $M = |E|$ be the number of edges in $G$. Nodes $i$ and $j$ can exchange messages with each other if and only if $(i,j) \in E$. Moreover, we denote by $N_{i} = \{j\,|\, (i,j) \in E\}$ be the neighboring set of node $i$. 
To find $\bar{c}$, distributed averaging-consensus methods have been studied extensively in the literature \cite{kia2019tutorial}. In particular, each node $i$ maintains a variable $x_{i}$ and initializes at its constant $c_i$, i.e., $x_{i}(0) = c_{i}$. The connectivity between the nodes gives an idea as to how the mixing of information between the nodes take place. For this we first introduce the notion of graph incidence matrix which instills a sense of orientation to each nodes. Thus, assuming that the $i^{\text{th}}$ and $j^{\text{th}}$ nodes are connected then each of these nodes share a relative information $x_{i} - x_{j}$. The orientation is determined by considering one of the nodes as the source and the other as a sink. Let $\mathbf{D} \in \mathbb{R}^{N \times M}$ represent the incidence matrix for the network with each entry $d_{ik}$ is defined as 
    \begin{align*}
d_{ik} = \left\{\begin{array}{ll}
+1 &  \text{if positive end of link $k$ is node $i$}\\
-1 & \text{if negative end of link $k$ is node $i$}\\
0 & \text{if no connection}.
\end{array}\right.    
\end{align*}
Here we note that above definition does not imply any directivity in the graph as the reader may choose the source and sink according to his/her convenience. The continuous-time version of distributed averaging-consensus methods then iteratively updates $x_{i}$ as  
\begin{align}\label{eq:dynamics_incidence}
    \dot{x}_{i} = -\sum_{k=1}^{M}d_{ik}\zeta_{k},
\end{align}
where $\zeta_{k}$ is the difference variable that captures the relative information between two nodes given by
\begin{align}
    \zeta_{k}=\sum_{l=1}^{N}d_{lk}x_{l} = \left\{\begin{array}{ll}
x_{i}-x_{j} &  \text{if $i$ is the positive end}\\
x_{j} - x_{i} & \text{if $j$ is the positive end}.
\end{array}\right.   \notag
\end{align}
    
The vector form of \eqref{eq:dynamics_incidence} is given as $\Dot{x} = -\mathbf{D}\mathbf{D}^{T}x,$ where $x = (x_{1},x_{2},\ldots,x_{N})^T$. We note the graph incidence matrix provides the mixing of information in the network with respect to the orientation associated with links between the nodes. Often it is more convenient to characterize the consensus dynamics just by knowing if any of two nodes are connected or not without the knowledge of any orientation. For this purpose, we introduce the graph Laplacian matrix 
$\mathbf{L} \in \mathbb{R}^{N \times N}$, whose $ij$-th entry is
    \begin{align}\label{eq:Laplacian_entry_def}
\ell_{ij} = \left\{\begin{array}{cc}
\text{deg}(i) &  \text{if} \ i=j,\\
-1 & \text{if }  i \neq j,\\
0 & \text{otherwise,} 
\end{array}\right.    
\end{align}
where $\text{deg}(i)$ is the degree of $i^{\text{th}}$ which represents the total number of edges that are incident on the node. The entries are zero if nodes are not connected. This implies that $\mathbf{L} = \mathbf{D}\mathbf{D}^T$, thus, we have  
\begin{align}\label{Eq:Consesus_dynamics_L}
    \Dot{x}  &= -\mathbf{L}x.
\end{align}
With this dynamics one can show that from \cite{Olfati_Saber_TAC2004} $\lim_{t\rightarrow\infty} x_{i}(t) = \bar{c},\quad \forall i\in V.$
\begin{figure}
    \centering
    \vspace{0.3cm}
    \includegraphics[width=0.75\columnwidth]{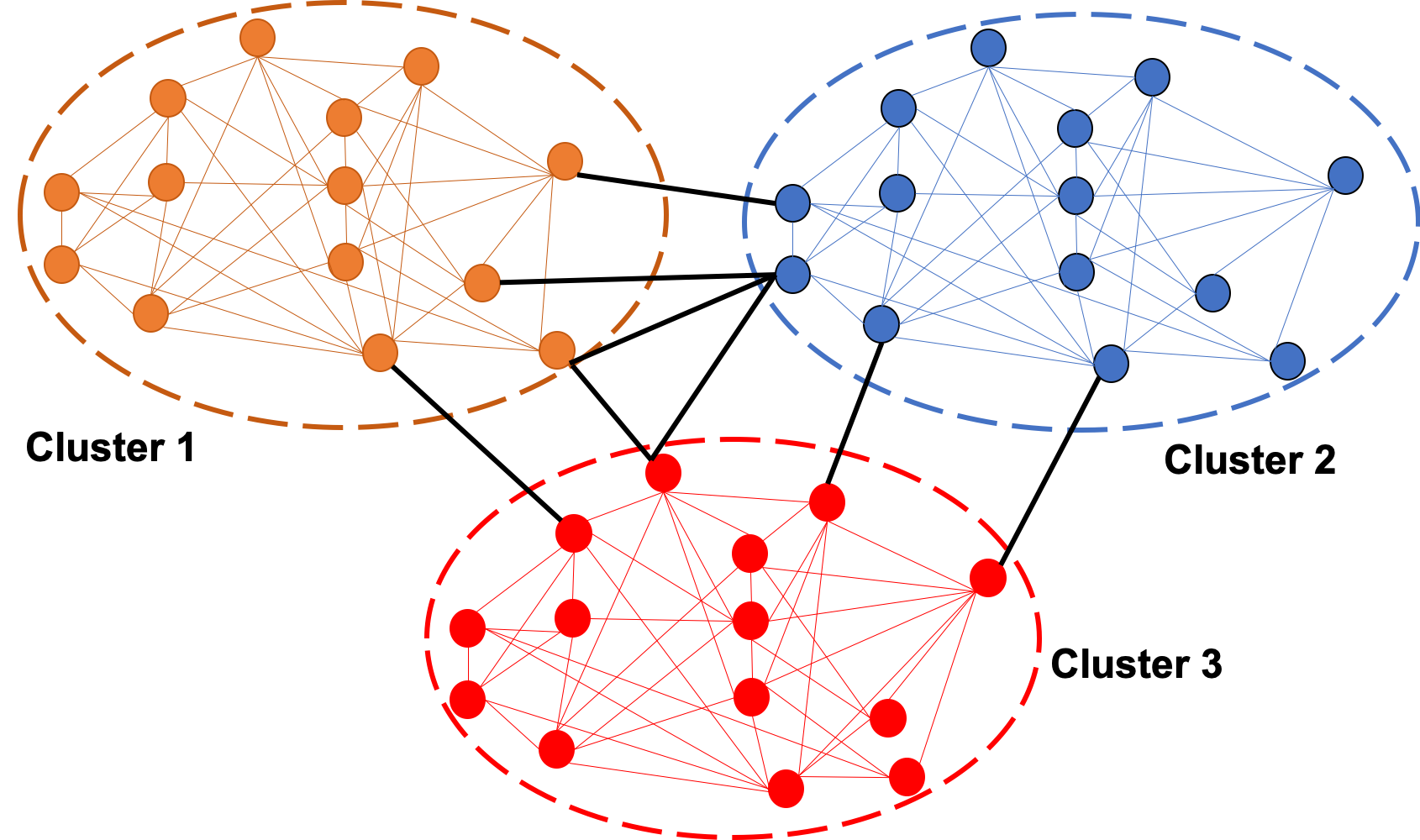}
    \caption{$42$-nodes network is partitioned into three densely connected clusters with sparse connections between them.}
    \label{fig:clusters}\vspace{-0.5cm}
\end{figure}

Note that $\mathbf{L}$ is a symmetric matrix. Moreover, when $G$ is connected,  the sum of every row of $\mathbf{L}$ is equal to zero. We also know from \cite{mesbahi2010graph} that $\mathbf{L}$ is positive semidefinite; hence its real eigenvalues can be ordered
as $\sigma_{1}(\mathbf{L}) \leq \sigma_{2}(\mathbf{L}) \leq ...\leq \sigma_{N}(\mathbf{L}),$
with $\sigma_{1}(\mathbf{L}) = 0,$ where $\sigma_{i}(\mathbf{L})$ is the $\text{i}^{\text{th}}$ eigenvalue of $\mathbf{L}$. 
In this paper, our focus is to study the performance of the distributed consensus-averaging method when the communication graph $G$ has a cluster structure which exhibits two-time-scale dynamics, where the information within any cluster is mixed much faster (due to dense communications) than information diffused from one cluster to another (due to sparse communications). Our primary goal is to understand the impact of the cluster structure on the performance of these methods, especially their convergence rates. In the next section we model the cluster network by characterizing the slow, fast and inter-cluster variables which play a significant role in the main analysis of this paper.

\section{Modeling of the Cluster Network Dynamics}
In this section, we characterize the network properties and define the slow and fast states. Towards this goal, we consider the case when $G$ is divided into $r$ densely connected clusters $C_{1},C_{2},\ldots,C_{r}$ while there are sparse connections between clusters. Fig. \ref{fig:clusters} illustrates such a graph where $r=3$ and communications between nodes within any cluster are dense while they are sparse across different clusters. Next, we denote by $N_{\alpha} = |C_{\alpha}|$ the number of nodes in cluster $\alpha$ and $\sum_{\alpha=1}^{r}N_{\alpha} = N$. We represent the total number of internal links associated with he clusters as $M^{I}$ and the total number of external links in the external graph as $M^{E}$. We also note that $M^{I} + M^{E} = M$ which is the total number of links.
The development in this section will help us to facilitate our analysis given in the next section. Our approach is motivated by the singular perturbation theory. However, we provide in this paper different formulations, relative to \cite{chow1985time} and \cite{biyik2008area} for the fast and slow dynamics. This results in a model that is not in the standard singularly perturbed form. We argue that this formulation makes it easier in studying the convergence rates of the two-time-scale behavior of distributed methods over cluster networks. We next define the consensus dynamics of a cluster network, which will be the basis for defining the slow and fast states. 

\begin{figure}[t!]
    \centering
    \includegraphics[width=\linewidth]{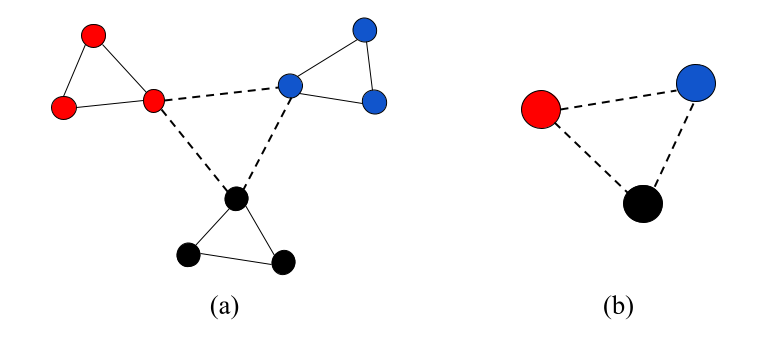}
    \caption{A $9$-node graph with $3$ clusters in (a). Graph for the aggregated network for (a) shown in (b).}\vspace{-0.3cm}
    \label{fig:eg_nw}      
\end{figure}
\subsection{Consensus Dynamics}
In this section, we present the formulation of the distributed consensus averaging method for a cluster network. For this  we consider $G_{\alpha} = (V_\alpha,E_{\alpha})$ be the graph representing the connection between nodes in cluster $C_{\alpha}$, for $\alpha = 1,\ldots, r$. In addition, let $G_{E}$ be the graph describing the external connections between clusters. Associated with each $C_{\alpha}$ is an incident matrix $\mathbf{D}_{\alpha}\in\mathbb{R}^{|V_{\alpha}|\times |E_{\alpha}|}$. Let $\mathbf{D}^{I}$ be $ \mathbf{D}^{I} = \text{diag}\{\mathbf{D}_{1}, \mathbf{D}_{2},...,\mathbf{D}_{r}\} \in \mathbb{R}^{N \times M^{I}}$ ,where $M^{I}$ be the total number of internal edges within the clusters, i.e., $M^I = \sum_{\alpha}|E_{\alpha}|$. We denote by $\mathbf{D}^{E}\in \mathbb{R}^{N\times M^{E}}$ the incidence matrix corresponding to the external graph $G_{E}$, where $M^E$ is the number of external edges connecting the clusters. Thus, we have $\mathbf{D} = [\mathbf{D}^{I}\,|\, \mathbf{D}^{E}] \in \mathbb{R}^{N\times M}$ and
\[\mathbf{D}\mathbf{D}^{T} = (\mathbf{D}^{I})(\mathbf{D}^{I})^{T} + (\mathbf{D}^{E})(\mathbf{D}^{E})^{T}.\] Similarly, we denote by $\mathbf{L}_{\alpha} = \mathbf{D}_{\alpha}\mathbf{D}_{\alpha}^T$ be the Laplacian matrix corresponding to $G_{\alpha}$ and let $\mathbf{L}^{I} = \text{diag}\{\mathbf{L}_{1}, \mathbf{L}_{2},...,\mathbf{L}_{r}\}$. In addition, let $\mathbf{L}^{E}$ be the external Laplacian matrix corresponding to $G_{E}$. Thus we have $\mathbf{D}\mathbf{D}^T = \mathbf{L}^{I} + \mathbf{L}^{E},$ which by \eqref{Eq:Consesus_dynamics_L} gives,
\begin{equation}\label{Eq:Consesus_dynamics_Laplacian}
    \dot{x}=-(\mathbf{L}^{I} + \mathbf{L}^{E})x.
\end{equation}
\begin{assump}\label{assump:graph_connectivity}
We assume that each $G_{\alpha}$, for all $\alpha = 1,\ldots, r$, and $G_{E}$ are connected and undirected. 
\end{assump}
\noindent\textit{An example of incidence and Laplacian matrices}: We present an example of a cluster network in Fig. \ref{fig:eg_nw} with $N=9$ nodes, $M=12$ edges and $r=3$ clusters Here we have a network of $r=3$ clusters. From the figure we have clusters $C_{1},C_{2},$ and $C_{3}$ represented by blue, black and red nodes respectively. The total number internal links is $M^{I} = 9$ and the total number of external links is $M^{E} = 3$. We observe that each cluster has an internal or local graph associated with it over which the nodes inside the clusters communicate with each other. The external communication graph is represent by green dashed edges. The clusters communicate with each other over this graph. For the internal incidence matrix we have  $\mathbf{D}^{I} = \text{diag}\{\mathbf{D}_{1}, \mathbf{D}_{2},\mathbf{D}_{3}\} \in \mathbb{R}^{9 \times 9}$ with 
\begin{align*}
    \mathbf{D}_{1} &= \begin{bmatrix}
    1 & 0 & -1\\
     -1 & 1 & 0\\
    0 & -1 & 1\\
    \end{bmatrix},\ \mathbf{D}_{2} = \begin{bmatrix}
    -1 & 1 & 0\\
     -1 & 0 & 1\\
    0 & -1 & 1\\
    \end{bmatrix},\\
    & \mathbf{D}_{3} = \begin{bmatrix}
    1 & 0 & -1\\
     -1 & 1 & 0\\
    0 & -1 & 1\\
    \end{bmatrix}.
\end{align*}
The external incidence matrix $\mathbf{D}_{E} \in \mathbb{R}^{9 \times 3}$ is given by
\begin{align*}
    \mathbf{D}_{E} &= \begin{bmatrix}
    1 & 0 & 0 & -1 & 0 & 0 & 0 & 0 & 0\\
    1 & 0 & 0 & 0 & 0 & 0 & -1 & 0 & 0\\
    0 & 0 & 0 & -1 & 0 & 0 & 1 & 0 & 0\\
    \end{bmatrix}^{T}.
\end{align*}
For the internal graph we have $\mathbf{L}^{I} = \text{diag}\{\mathbf{L}^{I}_{1}, \mathbf{L}^{I}_{2},\mathbf{L}^{I}_{3}\}$. Since all the clusters have similar topology, they have the same internal Laplacian matrix given by
\begin{align*}
    \mathbf{L}^{I}_{1} = \mathbf{L}^{I}_{2} = \mathbf{L}^{I}_{3} = \begin{bmatrix}
    2 & -1 & -1\\
     -1 & 2 & -1\\
    -1 & -1 & 2\\
    \end{bmatrix}.
\end{align*}
The external Laplacian matrix $L^{E}$is given by
\[
L^{E}=
\left[
\begin{array}{c|c|c}
L^{E}_{11} & L^{E}_{12} & L^{E}_{13} \\
\hline
L^{E}_{21} & L^{E}_{22} & L^{E}_{23} \\
\hline
L^{E}_{31} & L^{E}_{32} & L^{E}_{33} \\
\end{array}
\right].
\]
Since each node in the external graph has degree $=2$, $L^{E}_{11} = L^{E}_{22} = L^{E}_{33} = \text{diag}(2,0,0)$. Only the first node of each cluster communicates with other external nodes in this graph. Hence, we have 
\begin{align*}
     L^{E}_{12}= L^{E}_{13} = L^{E}_{21} = L^{E}_{23} = L^{E}_{31} = L^{E}_{32} = \begin{bmatrix}
    -1 & 0 & 0\\
     0 & 0 & 0\\
    0 & 0 & 0\\
    \end{bmatrix}.
\end{align*}
\subsection{Fast and slow variables within clusters}
As mentioned, distributed consensus methods over cluster networks often constitute two-time-scale dynamics. We define in this subsection explicit quantities to represent these dynamics, which are based on the difference in  the way information is mixed within any cluster versus the one diffused from one cluster to another cluster. We note our model of fast and slow variables introduced in this paper is fundamentally different from the ones used in \cite{biyik2008area} and \cite{chow1985time}.
\subsubsection{Slow variable}
To motivate our idea, let us assume that $G_{E}$ is empty, that is, the clusters $G_{\alpha}$ are disconnected from each other. In this case, each cluster is implementing an averaging method, which is expected to converge to the average of its internal nodes. On the other hand, when $G_{E}$ is nonempty the sparse interaction between clusters is to slowly push these averages to a common value. Thus, it is natural to define the slow variable in each cluster $\alpha$ as the average of its nodes' variables, i.e., let $\bar{x}_{\alpha}$ be the slow variable of $C_{\alpha}$ defined as $ \bar{x}_{\alpha} = \frac{1}{N_{\alpha}}\mathbf{1}_{N_\alpha}^{T}x^{\alpha},$
where $\mathbf{1}_{N_\alpha} = [1,1,...,1]^{T} \in\mathbb{R}^{N_{\alpha}}$ and $N_{\alpha}$ is the number of nodes in $C_\alpha$. We denote by $\mathbf{U}$ and $\mathbf{P}$, respectively, as 
\begin{align}
\begin{aligned}
\mathbf{U}  &=  diag(\mathbf{1}_{N_1},\mathbf{1}_{N_2},...\mathbf{1}_{N_r}) \in \mathbb{R}^{N \times r},\\ 
\mathbf{P} &= diag(N_{1},N_{2},...N_{r}) \in \mathbb{R}^{r \times r}.
\end{aligned}\label{notation:U_P}
\end{align}
Note that $\mathbf{P} = \mathbf{U}^{T}\mathbf{U}$. Moreover, let $y$ be defined as $y = [y_{1},y_{2},..y_{r}]^{T} \triangleq [\bar{x}_{1},\bar{x}_{2},\ldots,\bar{x}_{r}]^T \in \mathbb{R}^{r},      $
where $y_{\alpha} = \bar{x}_{\alpha}$. Then
we denote by $y$ the slow variable for the entire network
\begin{align}\label{Eq:slow_variable_nw}
    y = \mathbf{P}^{-1}\mathbf{U}^{T}x.
\end{align}
It follows from the formulation that the slow variable represents the states of the nodes of the aggregate network derived from the cluster network. Let the aggregate network be characterized by $\widetilde{G}_{E}$. We denote the Laplacian corresponding to the external aggregate graph as $\widetilde{\mathbf{L}}^{E} \in \mathbb{R}^{r \times r}$. Note that ${\mathbf{L}}^{E}$ is a sparse matrix due to the cluster structure of our graph. We consider the following useful result about the relationship between $\mathbf{L}^{E}$ and  $\widetilde{\mathbf{L}}^{E}$.
\begin{lem}\label{lem:L_tilde}

The Laplacian $\widetilde{\mathbf{L}}^{E}$ satisfies
\begin{align}
    \widetilde{\mathbf{L}}^{E} = \mathbf{U}^{T}\mathbf{L}^{E}\mathbf{U} .\label{lem:L_tilde:eq}
\end{align}
\end{lem}

\begin{proof}
Observe that 
\begin{align*}
   \mathbf{L}^{E} =  \begin{bmatrix}
    \mathbf{L}^{E}_{11} & \mathbf{L}^{E}_{12} & \cdots & \mathbf{L}^{E}_{1r}\\
    \mathbf{L}^{E}_{21} & \mathbf{L}^{E}_{22} & \cdots & \mathbf{L}^{E}_{2r}\\
    \vdots & \vdots & \ddots & \vdots\\
    \mathbf{L}^{E}_{r1} & \mathbf{L}^{E}_{r2} & \cdots & \mathbf{L}^{E}_{rr}
    \end{bmatrix},
\end{align*}
where $\mathbf{L}^{E}_{\alpha \alpha}$ is a diagonal matrix associated with the cluster $\alpha$ whose the diagonal elements are the external degree of each node (the number of external connections). Note that these diagonal elements can be equal to zero since most of the nodes may not have any external connection. In addition, the off-diagonal matrices represent the inter-cluster edges between clusters. Next using the structure of $\mathbf{U}$ from \eqref{notation:U_P} we have $\mathbf{U}^{T}\mathbf{L}^{E}\mathbf{U} = \Big[\mathbf{\mathbf{1}_{N_{\alpha}}^T\mathbf{L}_{\alpha\beta}^{E}\mathbf\mathbf{1}_{N_{\beta}}\Big]\in \mathbb{R}^{r\times r} }, \text{ for } \alpha,\beta = 1,\ldots,r.$
Here, $1^{T}_{N_{\alpha}}\mathbf{L}^{E}_{\alpha \alpha}\mathbf{1}_{N_{\alpha}}$ is the sum of all the external degrees  associated with the cluster $\alpha$. Also, the off-diagonal elements are the sum of the inter-cluster edges. Thus, $\mathbf{U}^{T}\mathbf{L}^{E}\mathbf{U}$ represents the Laplacian of $\widetilde{G}_{E}$.
\end{proof}
\subsubsection{Fast variable}
Since $G_{\alpha}$ is a dense graph, we expect that each $x_{i}^{\alpha}$ converges quickly to $\bar{x}_{\alpha}$. Thus, we define the fast variables within any cluster as the relative difference of its states to its slow variable. It can also be viewed as the residue or error of the state of a node with respect to the nodes' average in a cluster. Hence, we define the fast variable $e_{x_{i}}^{\alpha}$ in $C_{\alpha}$ as
\begin{align}\label{Eq:fast_variable_node}
    e_{x_i}^{\alpha} = x_{i}^{\alpha} - \bar{x}_{\alpha}. 
\end{align}
\begin{remark}
It is worth noting that the definition of the fast state \eqref{Eq:fast_variable_node} is different than that in \cite{chow1985time} and \cite{biyik2008area}, where the fast state is defined as the relative difference in state values with respect to a reference node (typically the first node). This results in a simpler representation of the system dynamics as there is no need to use a complex similarity transformation. In addition, the formulation \eqref{Eq:fast_variable_node} will help to characterize explicitly the rates of the algorithm, which may not be obvious to derive from the ones in  \cite{chow1985time,biyik2008area}.
\end{remark}
For each cluster $\alpha$, we denote by the diagonally dominant centering matrix as $\mathbf{W}_{\alpha} = \left(\mathbf{I}_{N_\alpha} - \frac{1}{N_{\alpha}}\mathbf{1}_{N_\alpha}\mathbf{1}_{N_\alpha}^{T}\right) \in \mathbb{R}^{N_{\alpha}\times N_{\alpha}},\label{notation:W_alpha} $
and let $e_{x}^{\alpha}$ be a vector in $\mathbb{R}^{N_{\alpha}}$, whose i-th entry is $e_{x_{i}}^{\alpha}$. In view of \eqref{Eq:fast_variable_node}, we have
\begin{align}
     e_{x}^{\alpha} &= x^{\alpha} - \bar{x}_{\alpha}\mathbf{1}_{N_\alpha} = x^{\alpha} - \frac{1}{N_{\alpha}}\mathbf{1}_{N_\alpha}^{T}x^{\alpha}\mathbf{1}_{N_\alpha}, \nonumber \\
    &= (\mathbf{I}_{N_\alpha} - \frac{1}{N_{\alpha}}\mathbf{1}_{N_\alpha}\mathbf{1}_{N_\alpha}^{T})x^{\alpha} = \mathbf{W}_{\alpha}x^{\alpha}. \label{Eq:fast_variable_area_centering}
\end{align}
Finally, we denote by $\mathbf{W} = diag (\mathbf{W}_{\alpha}) \in \mathbb{R}^{N\times N}$ and $e_{x}$ as $e_{x} = [({e_{x}}^{1})^{T},...,({e_{x}}^{r})^{T}]^T \in \mathbb{R}^{N}.$
Thus, the fast variable for the entire network $G$ is given by
\begin{align}
e_{x} &= x - \mathbf{U}y = \mathbf{W}x. \label{Eq:fast_variable_area_v2}
\end{align}
\subsubsection{Fast and Slow Dynamics}
We next present the dynamics for the fast and slow variables based on \eqref{Eq:slow_variable_nw} and \eqref{Eq:fast_variable_area_v2}.
\begin{lem}\label{lem:fast_slow_dynamics}
The fast and slow variables $e_{x}$ and $y$ satisfy 
\begin{align}
\dot{e}_{x} &= \frac{d e_{x}}{dt} =  -\mathbf{W}\mathbf{L}^{I}e_{x} -\mathbf{W}\mathbf{L}^{E}e_{x} - \mathbf{W}\mathbf{L}^{E}\mathbf{U}y,  \label{lem:fast_slow_dynamics:Eq_z}\\
\dot{y} &= \frac{dy}{dt} = - \mathbf{P}^{-1}\widetilde{\mathbf{L}}^{E}y - \mathbf{P}^{-1}\mathbf{U}^{T}\mathbf{L}^{E}e_{x}.  \label{lem:fast_slow_dynamics:Eq_y}
\end{align}
\end{lem}

\begin{proof}
We first show \eqref{lem:fast_slow_dynamics:Eq_z}. By using \eqref{Eq:fast_variable_area_v2} and \eqref{Eq:Consesus_dynamics_Laplacian} we have
\begin{align}
    \dot{e_{x}} = \frac{d e_{x}}{dt} &= \mathbf{W}\dot{x} =  - \mathbf{W}(\mathbf{L}^{I} + \mathbf{L}^{E})x,\notag\\
    &= -\mathbf{W}(\mathbf{L}^{I} + \mathbf{L}^{E})e_{x} - \mathbf{W}(\mathbf{L}^{I} + \mathbf{L}^{E})\mathbf{U}y.\notag
\end{align}
Since $\mathbf{L}^{I}\mathbf{U} = 0$, the preceding relation gives \eqref{lem:fast_slow_dynamics:Eq_z}. Next, we show \eqref{lem:fast_slow_dynamics:Eq_y}. In view of \eqref{Eq:slow_variable_nw} and \eqref{Eq:Consesus_dynamics_Laplacian}, we obtain
\begin{align}
    \dot{y} &= \mathbf{P}^{-1}\mathbf{U}^{T}\dot{x} = - \mathbf{P}^{-1}\mathbf{U}^{T}( \mathbf{L}^{I} + \mathbf{L}^{E})(e_{x} + \mathbf{U}y), \notag
\end{align}
which since $\mathbf{U}^T\mathbf{L}^{I} = \mathbf{L}^{I}\mathbf{U} = 0$ yields  
\begin{align}
     \dot{y} &= - \mathbf{P}^{-1}\mathbf{U}^{T}(\mathbf{L}^{I} + \mathbf{L}^{E})\mathbf{U}y - \mathbf{P}^{-1}\mathbf{U}^{T}(\mathbf{L}^{I} + \mathbf{L}^{E})e_{x}, \notag\\
     & = - \mathbf{P}^{-1}\mathbf{U}^{T}\mathbf{L}^{E}\mathbf{U}y - \mathbf{P}^{-1}\mathbf{U}^{T}\mathbf{L}^{E}e_{x}\notag,
\end{align}\label{Eq:slow_variable_dynamics_ver_3}
Next using Lemma \ref{lem:L_tilde} we have \eqref{lem:fast_slow_dynamics:Eq_y}.
\end{proof}

We observe that the dynamics of the slow variable depends on $\widetilde{\mathbf{L}}^{E}$. Since the slow variable can also be viewed as area aggregate variable, its states are expected to evolve according to the external aggregate communication graph $\widetilde{G}_{E}$. The second term in \eqref{lem:fast_slow_dynamics:Eq_y} can be seen as form of noise which depends on the fast variable, whose contribution will be prominent till the nodes inside the area synchronize. This happens relatively quickly and thus will have less contribution over a long term. For the consensus problem, we require the nodes in a cluster to converge to their average value of the initial condition. After local synchronization, the clusters behave as aggregate nodes, which eventually converge to the average of their initial conditions. This average coincides with the average of the initial conditions of all the nodes in the network. This motivates us to define a variable for each area that represents an error in consensus, and is expressed as $e_{y}^{\alpha} = y_{\alpha} - \bar{y},$ where $\bar{y} = \frac{1}{r}\sum_{\alpha = 1}^{r}y_{\alpha}$ is the average of the slow variables of each clusters which gives the overall average of all the states in the network. For the entire network, define $ e_y= [e_{y}^{1},...,e_{y}^{r}]^T \in \mathbb{R}^{r \times 1}$, which can also be expressed as 
\begin{align}\label{Eq:e_y_final}
    e_{y} &= y - \bar{y}\mathbf{1}_{r} = (\mathbf{I}_{r} - \frac{1}{r}\mathbf{1}_{r}\mathbf{1}_{r}^{T})y = \mathbf{W}_{r}y, 
\end{align}
    
where $\mathbf{1}_{r} = [1,1,...,1]^{T}\in \mathbb{R}^{r}$. Recall that $r$ is the number of areas. Here $\mathbf{W}_{r}$ is the centering matrix of dimension $r \times r$ associated with the external communication graph. 
The evolution of the inter-area states $e_{y}^{\alpha}$ is essential for our analysis as we are interested in how fast the aggregate nodes will converge to the consensus value. The dynamics of $e_{y}$ is given in the following lemma. 
\begin{lem}\label{lem:y_tilde_dynamics}
The inter-area state $e_y$ satisfies 
\begin{align}
\dot{e}_{y} = -\mathbf{W}_{r}\mathbf{P}^{-1}\widetilde{\mathbf{L}}^{E}e_{y}- \mathbf{W}_{r}\mathbf{P}^{-1}\mathbf{U}^{T}\mathbf{L}^{E}e_{x}.  \label{lem:y_tilde_dynamics:Eq}
\end{align}
\end{lem}

\begin{proof}
In view of  \eqref{Eq:e_y_final} and Lemma \ref{lem:fast_slow_dynamics}, we have 
     \begin{align}\label{eq:ey_dot_1}
         \dot{e}_{y} =- W_{r}\mathbf{P}^{-1}\mathbf{U}^{T}\mathbf{L}^{E}\mathbf{U}y -  W_{r}\mathbf{P}^{-1}\mathbf{U}^{T}\mathbf{L}^{E}e_{x}.
     \end{align}
    We  first observe that $\widetilde{\mathbf{L}}^{E}\mathbf{1}_{r}\mathbf{1}_{r}^{T} = \mathbf{0}$ since $\mathbf{1}_{r}\mathbf{1}_{r}^{T}$ is a matrix whose colums are vector of ones of dimension $r \times 1$ each of which are in the Null space of $\widetilde{\mathbf{L}}^{E}$. Thus the first term in the above equation can be expressed as 
\begin{align}
    &- \mathbf{W}_{r}\mathbf{P}^{-1}\widetilde{\mathbf{L}}^{E}y = -\mathbf{W}_{r} \mathbf{P}^{-1}\widetilde{\mathbf{L}}^{E}(\mathbf{I}_{r} - \frac{1}{r}\mathbf{1}_{r}\mathbf{1}_{r}^{T})y, \nonumber \\
    & = - \mathbf{W}_{r}\mathbf{P}^{-1}\widetilde{\mathbf{L}}^{E}\mathbf{W}_{r}y =  -\mathbf{W}_{r}\mathbf{P}^{-1}\widetilde{\mathbf{L}}^{E}e_{y}.\notag
\end{align}
Putting this back in \eqref{eq:ey_dot_1} we obtain \eqref{lem:y_tilde_dynamics:Eq}.
\end{proof}





\section{Main results}
In this section, we present the main results for our analysis of the convergence rate of the consensus problem over a cluster network. Our interest is in the dynamics of the inter-area variable which is associated with the slow time-scale and the fast variable that describes the evolution of the network dynamics in the fast time-scale. 
For our analysis we consider quadratic Lyapunov functions for the inter-area slow variable and the fast variable respectively given by $V_{e_{y}} = \frac{1}{2}\|e_{y}\|^2 \ \text{and} \ V_{e_{x}} = \frac{1}{2}\|e_{x}\|^2$. The convergence rate for the consensus of the entire network can be derived by considering an aggregate Lyapunov function given as 
\begin{align}\label{Eq:aggregate_lyp}
V = V_{e_{y}} + \epsilon V_{e_{x}}.     
\end{align}
Here $\epsilon >0$ can be chosen conveniently depending upon the network structure as shown in Theorem \ref{thm:covergence_rate}. This type of Lyapunov functions has been used recently in studying the convergence rates of the classic two-time-scale stochastic approximation; see for example \cite{Doan_two_time_SA2020,Doan_two_time_SA2021}. Also, we denote $N_{\min}$ and $N_{\max}$, respectively, as 
\begin{align}
N_{\min} = \min_{\alpha} N_{\alpha},\quad N_{\max} = \max_{\alpha}N_{\alpha}. \label{notation:Nmin_max}   
\end{align}
We denote 
\begin{align}
    \sigma_{2}^{I} = \min_{\alpha} \sigma_{2}(\mathbf{L}_{\alpha}),
\end{align}
which is strictly positive since the cluster $C_{\alpha}$ is connected. Also, let $\sigma_{2}(\widetilde{\mathbf{L}}^{E})$ be the second smallest eigenvalue of $\widetilde{\mathbf{L}}^{E}$, which is also strictly positive since $\widetilde{G}_{E}$ is connected.
Finally, since $\mathbf{P}^{-1}$ is a diagonal matrix with entries $\frac{1}{N_{\alpha}}$ as diagonal elements, we have
\begin{align}\label{eq:inequality_p_inv}
     \frac{1}{N_{\max}}\mathbf{I} \leq \mathbf{P}^{-1} \leq \frac{1}{N_{\min}}\mathbf{I}, \ N_{\min}\mathbf{I} \leq \mathbf{P} \leq N_{\max}\mathbf{I} 
\end{align}

We next consider the following useful lemmas about the Lyapunov functions $V_{e_{x}}$ and $V_{e_{y}}$. For convenience, we present the proofs of these lemmas in the Appendix. 
\begin{lem}\label{lem:V_ey_dot}
The Lyapunov function $V_{e_{y}}$ satisfies 
\begin{align}
    \frac{dV_{e_{y}}}{dt} \leq -\frac{1}{2}e_{y}^T\mathbf{P}^{-1}\widetilde{\mathbf{L}}^{E}e_{y} + \frac{1}{2N_{\min}}e_{x}^T\mathbf{L}^{E}e_{x}.\label{lem:V_ey_dot:ineq}
\end{align}
\end{lem}

\begin{lem}\label{lem:V_ex_dot}
The Lyapunov function $V_{e_{x}}$ satisfies 
\begin{align}
\frac{dV_{e_{x}}}{dt} \leq -e_{x}^T\mathbf{W}\mathbf{L}^{I}e_{x} - \frac{1}{2}e_{x}^T\mathbf{W}\mathbf{L}^{E}e_{x} +  \frac{1}{2}e_{y}^T\widetilde{\mathbf{L}}^{E}e_{y}.\label{lem:V_ex_dot:ineq}
\end{align}
\end{lem}

\begin{assump}\label{assump:G}
$G$ has a cluster structure that satisfies
\begin{align}
\sigma_{2}^{I}  \geq \frac{2\|\widetilde{\mathbf{L}}^{E}\|^{2}N_{\max}}{\sigma_{2}(\widetilde{\mathbf{L}}^{E})N_{\min}}.\label{assump:G:condition} 
\end{align}
\end{assump}
\begin{remark}
This assumption, similar to the one in \cite{biyik2008area} (see Section $4.2$), basically implies that the internal connections within any cluster is much denser than the external connections across clusters. In Section \ref{sec:simulations}, we verify that this condition holds in our simulations. 
\end{remark}
\begin{thm}\label{thm:covergence_rate}
Suppose that Assumptions \ref{assump:graph_connectivity} and \ref{assump:G} hold. Let $x(t)$ be updated by \eqref{Eq:Consesus_dynamics_Laplacian} and $\epsilon$ be chosen as 
\begin{equation}
\epsilon =    \frac{\sigma_{2}(\widetilde{\mathbf{L}}^{E})}{2N_{\max}\|\widetilde{\mathbf{L}}^{E}\|}\cdot \label{thm:covergence_rate:epsilon}
\end{equation}
Then we have
\begin{align}
V(t) \leq e^{-\frac{\sigma_{2}(\widetilde{\mathbf{L}}^{E})}{2N_{\max}}t}V(0).    \label{thm:covergence_rate:Ineq}
\end{align}
\end{thm}
\begin{remark}
1. Our result in \eqref{thm:covergence_rate:Ineq} implies that the distributed consensus method converges exponentially. In addition, the rate of convergence only depends on the external graph $\widetilde{G}_{E}$ through $\sigma_{2}(\widetilde{\mathbf{L}}^{E})$, which only scales with a small number of nodes since there are a few number of communications between clusters. On the other hand, the existing analysis for the rate of distributed consensus method depends on the total number of nodes $N$ when applying to the entire network. \\
2. The parameter $\epsilon$ can be viewed as a singular-perturbation parameter representing the time-scale difference in the fast and slow dynamics, similar to the one used in \cite{biyik2008area}.
\end{remark}
\begin{proof} 
We begin our proof by recalling the aggregate Lyapunov function  from \eqref{Eq:aggregate_lyp} given by $V = V_{e_{y}} + \epsilon V_{e_{x}},$ which by using \eqref{lem:V_ey_dot:ineq} and \eqref{lem:V_ex_dot:ineq} we have
\begin{align}
&\dot{V} = \dot{V}_{e_{y}} + \epsilon \dot{V}_{e_{x}} \leq  -\frac{1}{2}e_{y}^T\mathbf{P}^{-1}\widetilde{\mathbf{L}}^{E}e_{y} + \frac{1}{2N_{\min}}e_{x}^T\mathbf{L}^{E}e_{x}\notag\\
&\quad - \epsilon e_{x}^T\mathbf{W}\mathbf{L}^{I}e_{x} - \frac{ \epsilon}{2}e_{x}^T\mathbf{W}\mathbf{L}^{E}e_{x} +  \frac{\epsilon}{2}e_{y}^T\widetilde{\mathbf{L}}^{E}e_{y}. \label{thm:covergence_rate:Eq1}
\end{align}
We simplify the above inequality by bounding each term using the inequalities that we have established before. First, using the definition of $\mathbf{P}$ in \eqref{notation:U_P} and $N_{\max}$ in \eqref{notation:Nmin_max} yields $-\frac{1}{2}e_{y}^T\mathbf{P}^{-1}\widetilde{\mathbf{L}}^{E}e_{y} \leq -\frac{1}{2N_{\max}}e_{y}^T\widetilde{\mathbf{L}}^{E}e_{y}$. Here we observe that the Laplacian $\widetilde{\mathbf{L}}^{E}$ has a zero eigen value corresponding to the eigen vector $\mathbf{1}_{r}$ while the other eigen values are strictly positive.  
Recall that $\mathbf{1}_{r}$ is in the null space of $\mathbf{W}_{r}$, i.e., $\mathbf{1}_{r}^T\mathbf{W}_{r} = \mathbf{W}_{r}\mathbf{1}_{r} = 0$, and since $e_{y} = \mathbf{W}_{r}y$  we have that $e_{y}\in span\{\mathbf{1_{r}^{\perp}}\}$. Thus we have $- e_{y}^T\Tilde{\mathbf{L}}^{E}e_{y} \leq -\sigma_{2}(\Tilde{\mathbf{L}}^{E})\|e_{y}\|^2$, which in view of \eqref{eq:inequality_p_inv} implies that
\begin{align}\label{thm:covergence_rate:Eq1a} 
-\frac{1}{2}e_{y}^T\mathbf{P}^{-1}\widetilde{\mathbf{L}}^{E}e_{y} \leq -\frac{\sigma_{2}(\widetilde{\mathbf{L}}^{E})}{2N_{\max}}\|e_{y}\|^2.    
\end{align}
Since $\mathbf{L}^{E}$ is positive semi-definite  and $e_{x}^T\mathbf{W} = x^T\mathbf{W} \mathbf{W} = x^T\mathbf{W} = e_{x}^T $ hence we have
\begin{align}
\begin{aligned}
- \frac{ \epsilon}{2}e_{x}^T\mathbf{W}\mathbf{L}^{E}e_{x}  = - \frac{ \epsilon}{2}e_{x}^T\mathbf{L}^{E}e_{x} \leq 0
\end{aligned}.\label{thm:covergence_rate:Eq1b} 
\end{align}
Next considering the third term in the inequality \eqref{thm:covergence_rate:Eq1} we have $ - \epsilon e_{x}^T\mathbf{W}\mathbf{L}^{I}e_{x} =  - \epsilon e_{x}^T\mathbf{L}^{I}e_{x}$.
We note that the fast variable can be expressed as sub-vector corresponding to each cluster as $e_{x} = [(e_{x}^{1})^{T}...(e_{x}^{r})^{T}]^{T}$. Following this we have
\begin{align}
   & - \epsilon e_{x}^T\mathbf{L}^{I}e_{x} = - \epsilon \sum_{\alpha=1}^{r}(e_{x}^{\alpha})^{T}\mathbf{L}_{\alpha}^{I}e_{x}^{\alpha}  \leq - \epsilon \sum_{\alpha=1}^{r}\sigma_{2}(\mathbf{L}_{\alpha}^{I})||e_{x}^{\alpha}||^{2} \notag\\
    & \leq  - \epsilon \min_{\alpha}\sigma_{2}(\mathbf{L}_{\alpha}^{I})\sum_{\alpha=1}^{r}\| e_{x}^{\alpha}\|^2 =  - \epsilon \sigma_{2}^{I}\| e_{x}\|^2. \label{third_inequality}
\end{align}
Next we observe that since the external Laplacian matrix is sparse then we have $\|\Tilde{\mathbf{L}}^{E}\| \leq \|\mathbf{L}^{E}\|$. Next we simplify the second and last term from \eqref{thm:covergence_rate:Eq1}, so that we have 
\begin{align}\label{eq:covergence_rate_1c} 
\frac{1}{2N_{\min}}e_{x}^T\mathbf{L}^{E}e_{x} &\leq \frac{\|\Tilde{\mathbf{L}}^{E}\|}{2N_{\min}}\|e_{x}\|^2,\notag\\   \frac{\epsilon}{2}e_{y}^T\mathbf{U}^T\mathbf{L}^E\mathbf{U}e_{y} &\leq \frac{\|\Tilde{\mathbf{L}}^{E}\|\epsilon}{2}\|e_{y}\|^2 .   
\end{align}
By using \eqref{thm:covergence_rate:Eq1a}--\eqref{eq:covergence_rate_1c} into \eqref{thm:covergence_rate:Eq1} we have 
\begin{align}
    \dot{V} &\leq -\frac{\sigma_{2}(\Tilde{\mathbf{L}}^{E})}{2N_{\max}}\|e_{y}\|^2  + \frac{\|\Tilde{\mathbf{L}}^{E}\|}{2N_{\min}}\|e_{x}\|^2 - \epsilon \sigma_{2}^{I}\| e_{x}\|^2 \notag\\
    &+ \frac{\|\Tilde{\mathbf{L}}^{E}\|\epsilon}{2}\|e_{y}\|^2 = -\left(\frac{\sigma_{2}(\Tilde{\mathbf{L}}^{E})}{2N_{\max}}-\frac{\|\Tilde{\mathbf{L}}^{E}\|\epsilon}{2}  \right)\| e_{y}\|^{2} \notag\\
&- \left(\epsilon \sigma_{2}^{I} -  \frac{\|\Tilde{\mathbf{L}}^{E}\|}{2N_{\min}}\right)\| e_{x}\|^2. \label{eq:V_dot_ineq_1}
\end{align}
We note that to guarantee stability we need $\dot{V} < 0$.  Accordingly, choosing $\epsilon =    \frac{\sigma_{2}(\Tilde{\mathbf{L}}^{E})}{2N_{\max}\|\Tilde{\mathbf{L}}^{E}\|} $ leads to
\begin{align}
    \dot{V} &\leq -\frac{\sigma_{2}(\Tilde{\mathbf{L}}^{E})}{2N_{\max}}V_{e_{y}} - \left(\frac{\sigma_{2}^{I} \sigma_{2}(\Tilde{\mathbf{L}}^{E})}{N_{\max}\|\Tilde{\mathbf{L}}^{E}\|} 
   -  \frac{\|\Tilde{\mathbf{L}}^{E}\|}{N_{\min}}\right)V_{e_{x}} . \label{thm:covergence_rate:Eq2} 
\end{align}
Using Assumption \ref{assump:G} we have 
\begin{align}
    \frac{\sigma_{2}^{I} \sigma_{2}(\Tilde{\mathbf{L}}^{E})}{N_{\max}\|\Tilde{\mathbf{L}}^{E}\|} 
   -  \frac{\|\Tilde{\mathbf{L}}^{E}\|}{N_{\min}} \geq  \frac{\|\Tilde{\mathbf{L}}^{E}\|}{N_{\min}}\notag,
\end{align}
which when substituting \eqref{thm:covergence_rate:Eq2} in yields
\begin{align}
    \dot{V} &\leq -\frac{\sigma_{2}(\Tilde{\mathbf{L}}^{E})}{2N_{\max}}V_{e_{y}} -\frac{\|\Tilde{\mathbf{L}}^{E}\|}{N_{\min}}V_{e_{x}} \notag\\
    & = -\frac{\sigma_{2}(\Tilde{\mathbf{L}}^{E})}{2N_{\max}}V_{e_{y}} - \frac{\sigma_{2}(\Tilde{\mathbf{L}}^{E})\epsilon}{2N_{\max}}V_{e_{x}}\notag\\ &-\frac{\|\Tilde{\mathbf{L}}^{E}\|}{N_{\min}}V_{e_{x}} + \frac{\sigma_{2}(\Tilde{\mathbf{L}}^{E})\epsilon}{2N_{\max}}V_{e_{x}}\notag\\
    & = -\frac{\sigma_{2}(\Tilde{\mathbf{L}}^{E})}{2N_{\max}}(V_{e_{y}}+\epsilon V_{e_{x}})\notag\\
    &-\left( \frac{\|\Tilde{\mathbf{L}}^{E}\|}{N_{\min}}- \frac{\sigma_{2}^{2}(\Tilde{\mathbf{L}}^{E})}{4N_{\max}^{2}\|\Tilde{\mathbf{L}}^{E}\|^{2}}\right)V_{e_{x}}. \label{thm:covergence_rate:Eq3} 
\end{align}
Here we note have $\|\Tilde{\mathbf{L}}^{E}\| > \sigma_{2}(\Tilde{\mathbf{L}}^{E})$, hence we have a non-positive coefficient of $V_{e_{x}}$. Thus using \eqref{Eq:aggregate_lyp} leads us to the following
\begin{align}
\dot{V} &\leq -\frac{\sigma_{2}(\widetilde{\mathbf{L}}^{E})}{2N_{\max}}V. \label{thm:covergence_rate:Eq3} 
\end{align}
This immediately gives \eqref{thm:covergence_rate:Ineq} which concludes our proof. 
\end{proof}

\section{Simulation results}\label{sec:simulations}
In this section, we illustrate our theoretical results by a number of numerical experiments on distributed consensus methods over different cluster networks. Here we investigate the following:
\begin{itemize}
    \item First, we study how the convergence of the consensus dynamics depends upon changing the connectivity between the clusters.\label{itm:goal1}
    \item Second, we study how network scalability impacts the the convergence of the network. \label{itm:goal2}
\end{itemize}
\begin{figure}[h!]
    \centering
    \includegraphics[width=0.7\linewidth]{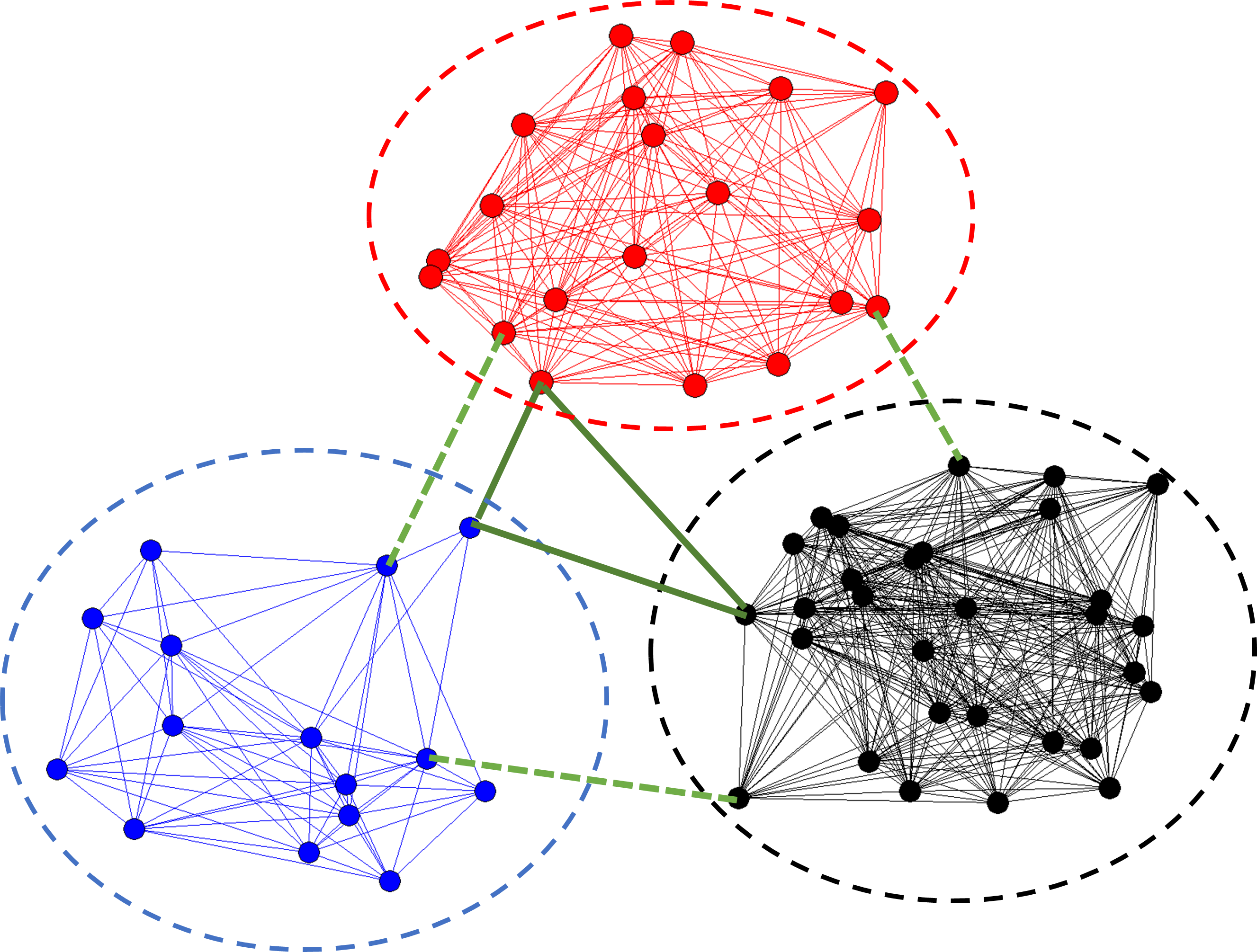}
    \caption{A network with 3 clusters composed of 15 (blue), 30 (black), and 20 (red) node.}
    \label{fig:Small_nw}     
\end{figure}

\subsection{Small Network }
\begin{figure}[h!]
    \centering
    \includegraphics[width= 0.75\linewidth]{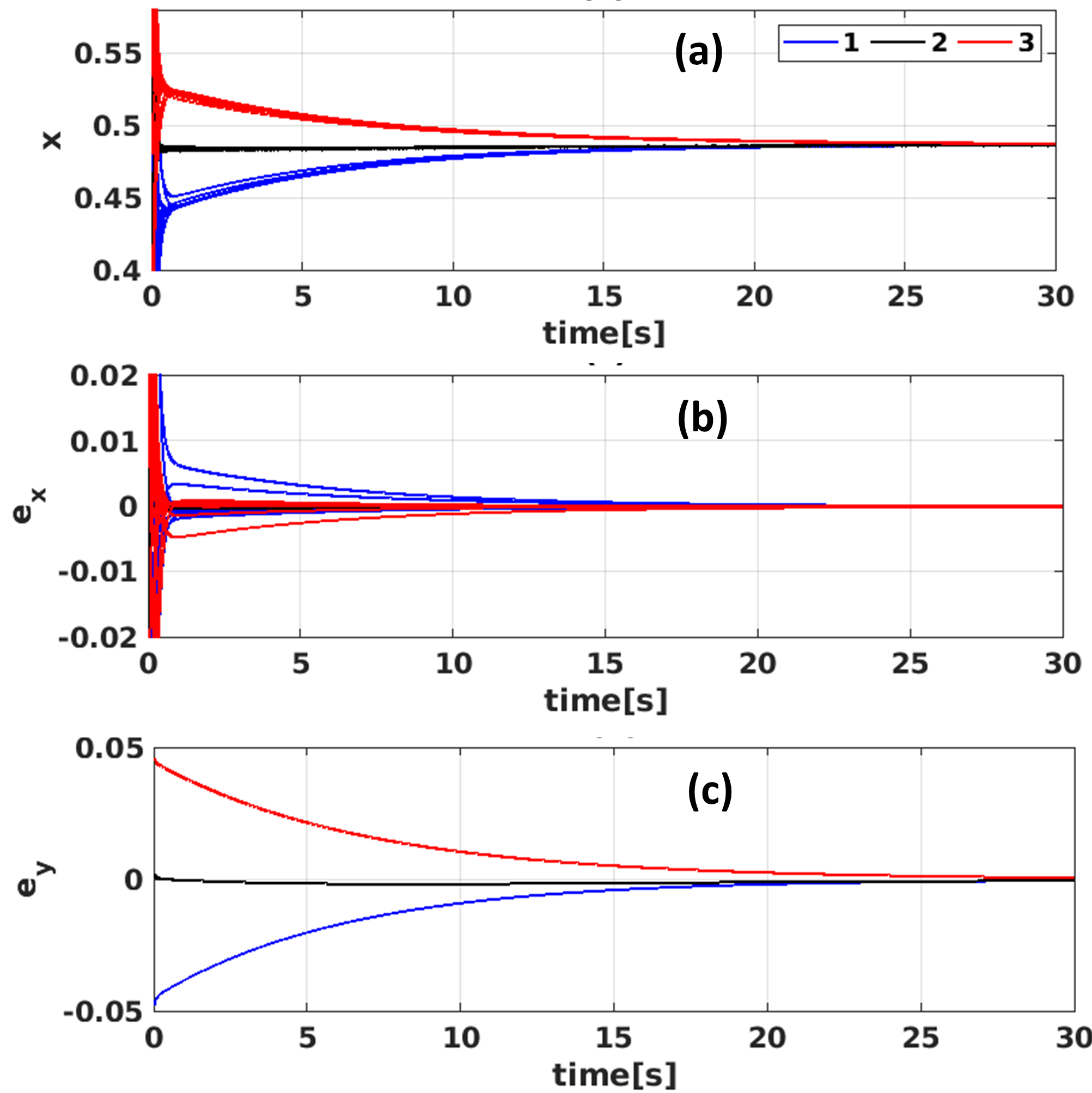}
    \caption{\textit{Simulation 1}: evolution of different variables for the network in Fig.\ \ref{fig:Small_nw} using 3 external communication links.} 
    \label{fig:FL_nw_2}     
\end{figure} 
\begin{figure}[h!]
    \centering
    \includegraphics[width= 0.75\linewidth]{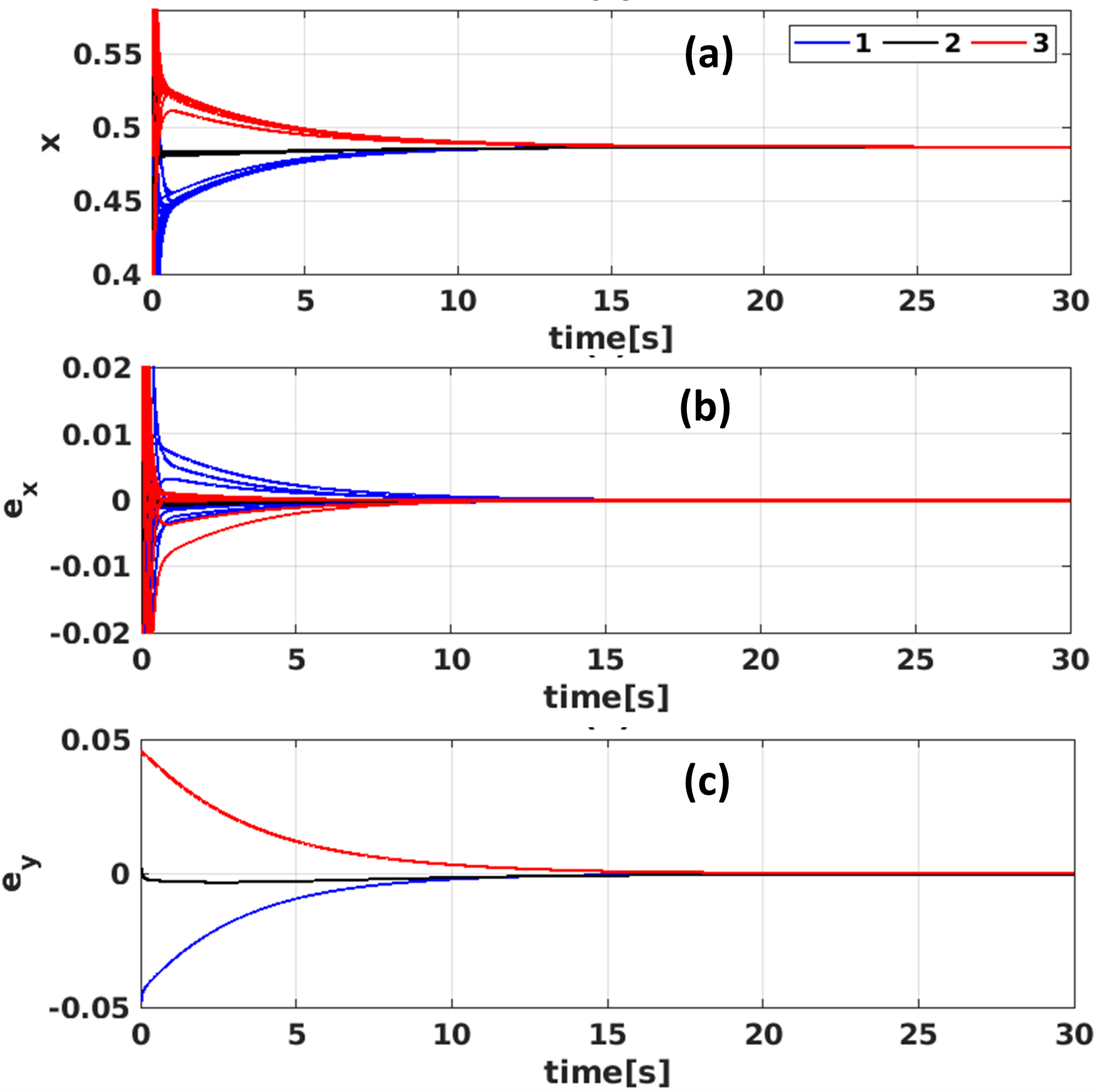}
    \caption{\textit{Simulation 2}: evolution of different variables for the network in Fig.\  \ref{fig:Small_nw} using $6$ external communication links.} 
    \label{fig:FL_nw_3}     \vspace{-0.3cm}
\end{figure}
For our simulation we consider a cluster network of $65$ nodes divided into $3$ distinct clusters and presented in Fig. \ref{fig:Small_nw}. We simulate two different networks, one with only three external communication edges as shown by green dashed edges, and the other with a total of $6$ communication edge as shown by both the dashed and solid edges. 

We first verify the condition in Assumption \ref{assump:G}. The network in Fig. \ref{fig:Small_nw} has $N_{\min} = 15$ and $N_{\max} =30$. The eigenvalues of the different Laplacian matrices are $\sigma_{2}(\mathbf{L}^{I}_{1}) = 6.79,\, \sigma_{2}(\mathbf{L}^{I}_{2}) = 16.48, \sigma_{2}(\mathbf{L}^{I}_{3}) = 6.52$ and $\sigma_{2}(\widetilde{\mathbf{L}}^{E}) = 3$. It is straight forward to verify that this network satisfies condition \eqref{assump:G:condition} in given in Assumption \eqref{assump:G}. 

We next present our simulations in Fig.\ \ref{fig:FL_nw_2}  (where we consider 3 external communication edges) and in Fig.\ \ref{fig:FL_nw_3} (where there are $6$ external communication edges). For our simulations, we have $\epsilon = 0.017$. In addition, the nodes' initial values are randomly initialized. We use \eqref{eq:Laplacian_entry_def} to generate $\mathbf{L}^{I}$ and $\mathbf{L}^{E}$. The two time-scale nature of the network can be observed in Fig. \ref{fig:FL_nw_2}, which shows the time response of the states $x$, the inter-area variable $e_{y}$, and the fast variable $e_{x}$. It can be seen from Fig. \ref{fig:FL_nw_2} (b) and (c) that the nodes perform first a local synchronization, then they slowly drive the entire network to the average of their initial conditions, which agrees with our theoretical results. Fig. \ref{fig:FL_nw_3} shows a similar result for the same network but using 6 external communication edges. Moreover, it also shows that the convergence of the variables is faster when the number of edges of the external graph increases as compared to the ones in Fig. \ref{fig:FL_nw_2}.  
Our simulations show that we can study how the states of all the $65$ nodes evolve by considering the dynamics of the inter-area variables, which has a much smaller dimension given by the number of clusters. This agrees with our theoretical result in Theorem \ref{thm:covergence_rate}.


\begin{figure}[t!]
    \centering
    \includegraphics[width=0.7\linewidth]{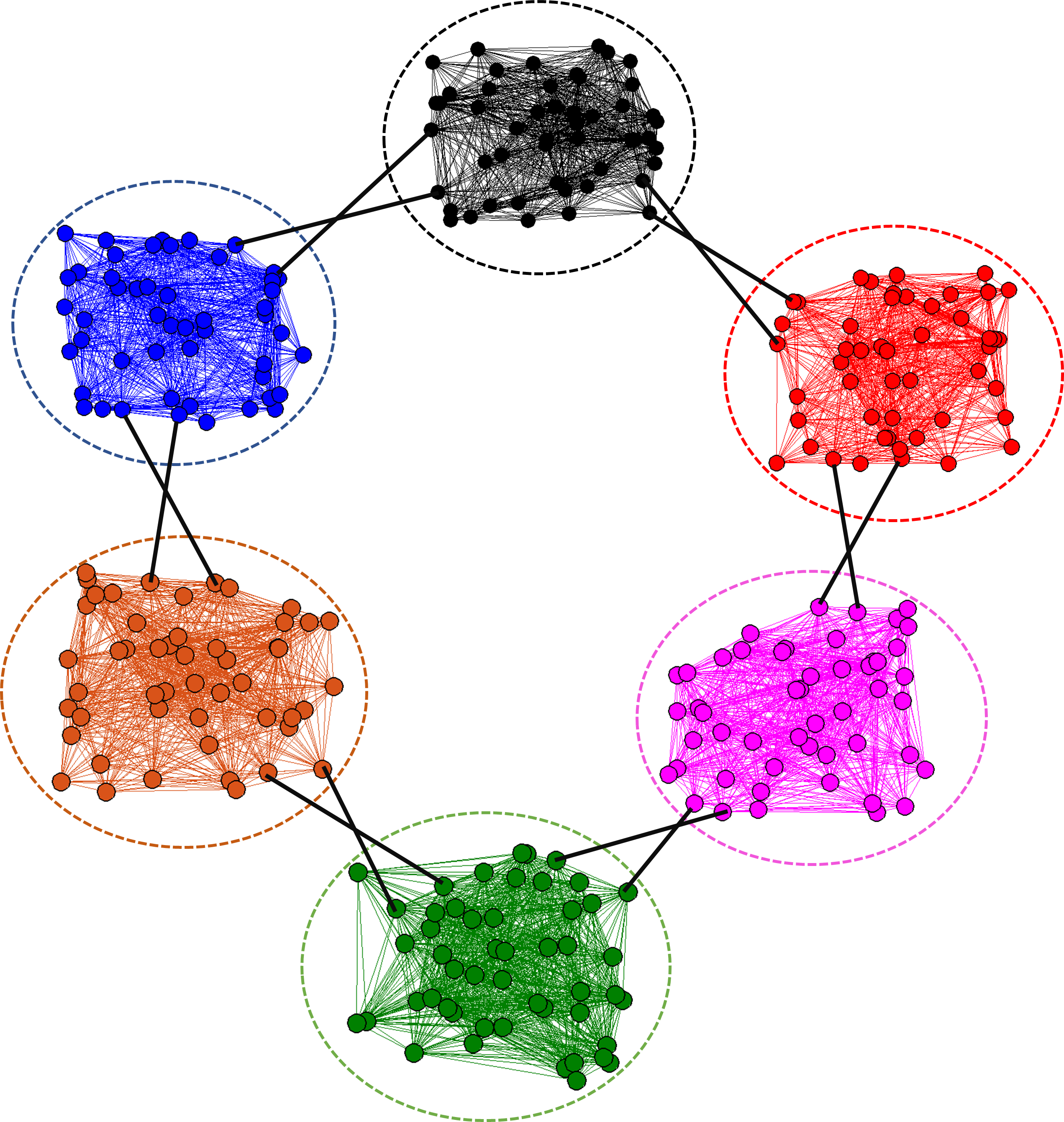}
    \caption{A $300$-nodes network divided equally into 6 clusters.}
    \label{fig:Large_nw}     
\end{figure}
\subsection{Large Network}
We now provide simulations on the performance of distributed consensus over large cluster networks where we intend to observe the convergence of the different network variables as we increase the number of clusters. Here we consider $6$ clusters with $50$ nodes in each cluster as shown in Fig. \ref{fig:Large_nw}. The individual clusters are generated in a similar manner to the simulations above. We consider that any two clusters communicate with each other using only two links. It is worth noting that in this case the fast-variable $e_{x}$ is a vector composed of 300 variables, the inter-area variable $e_{y}$ is a vector composed of 6 variables, and the Laplacian matrix is a square matrix with dimension $300$. The simulation results of this experiment is shown in Fig. \ref{fig:Large_nw_sim}, where we have the same observation as the ones in small networks. 
\begin{figure}[t!]
    \centering
    \includegraphics[width=0.75\linewidth]{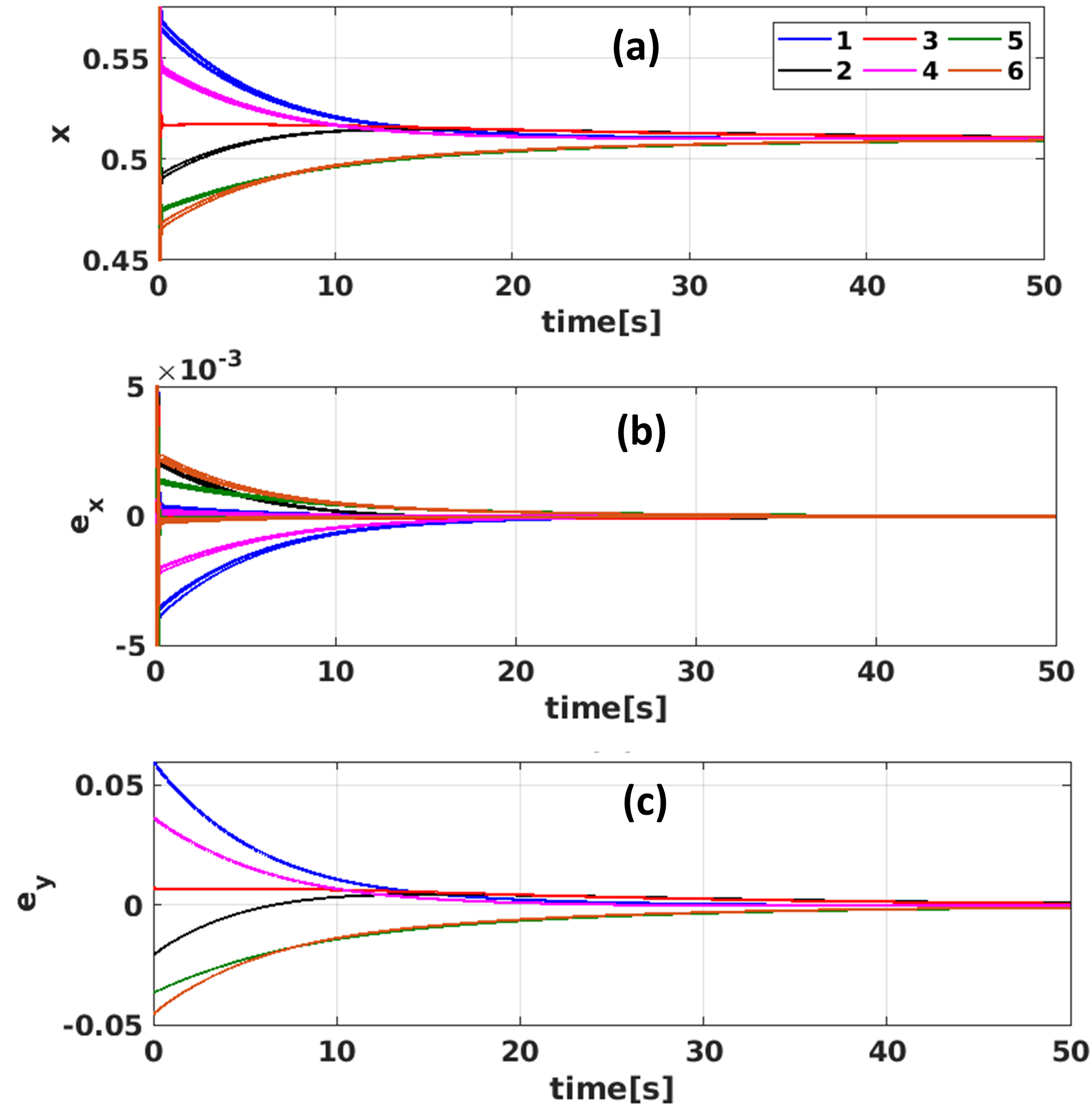}
    \caption{\textit{Simulation 3}: evolution of different variables for the network in Fig. \ref{fig:Large_nw}.}
    \label{fig:Large_nw_sim}     
\end{figure} 

\section{Appendix}
\subsection{Proof of Lemma \ref{lem:V_ey_dot}}
\begin{proof}
From \eqref{lem:y_tilde_dynamics:Eq} we have
\begin{align}
&\frac{d V_{e_{y}}}{dt} = e_{y}^T\dot{e_{y}} =-e_{y}^T\mathbf{W}_{r}\mathbf{P}^{-1}\widetilde{\mathbf{L}}^{E}e_{y} - e_{y}^T\mathbf{W}_{r}\mathbf{P}^{-1}\mathbf{U}^{T}\mathbf{L}^{E}e_{x} . \notag
\end{align}
Note that $e_{y}^T\mathbf{W}_{r} = y^T \mathbf{W}_{r}\mathbf{W}_{r} = y^T\mathbf{W}_{r} = e_{y}^T$, which gives
\begin{align}\label{eq:V_dot_e_y_v0}
    \frac{d V_{e_{y}}}{dt} = -e_{y}^T\mathbf{P}^{-1}\widetilde{\mathbf{L}}^{E}e_{y} - e_{y}^T\mathbf{P}^{-1}\mathbf{U}^{T}\mathbf{L}^{E}e_{x} .
\end{align}

For our analysis of the second term in \eqref{eq:V_dot_e_y_v0} we consider that for any two vectors $x, y \neq 0$ and a scalar $\eta > 0$, the following inequality holds
\begin{align}\label{eq:inequality_AMGM}
    2x^{T}y \leq 2||x||||y|| \leq \eta ||x||^{2} + \frac{1}{\eta}||y||^{2}
\end{align}
Here we conveniently choose $\eta = N_{min}$ and have,
\begin{align}
-& e_{y}^T\mathbf{P}^{-1}\mathbf{U}^{T}\mathbf{L}^{E}e_{x}    =     -\Big(e_{y}^T\mathbf{P}^{-1}\mathbf{U}^{T}\sqrt{\mathbf{L}^{E}}\Big)\Big(\sqrt{\mathbf{L}^{E}}e_{x}\Big), \notag\\
&\leq \frac{N_{\min}}{2}\|\sqrt{\mathbf{L}^{E}} \mathbf{U} \mathbf{P}^{-1}e_{y}\|^2 + \frac{1}{2N_{\min}} \|\sqrt{\mathbf{L}^{E}}e_{x}\|^2,\notag\\
&= \frac{N_{\min}}{2}e_{y}^T\mathbf{P}^{-1}\widetilde{\mathbf{L}}^{E}\mathbf{P}^{-1}e_{y} + \frac{1}{2N_{\min}}e_{x}^T\mathbf{L}^{E}e_{x}.\label{eq:cross_term_e_x}
\end{align}
From the above equation using \eqref{eq:inequality_p_inv}, we have 
\begin{align}
    \frac{N_{\min}}{2}e_{y}^T\mathbf{P}^{-1}\widetilde{\mathbf{L}}^{E}\mathbf{P}^{-1}e_{y} \leq \frac{1}{2}e_{y}^T\mathbf{P}^{-1}\widetilde{\mathbf{L}}^{E}e_{y}.
\end{align}
Thus from \eqref{eq:V_dot_e_y_v0} and \eqref{eq:cross_term_e_x} we have,
\begin{align*}
\frac{d V_{e_{y}}}{dt} \leq -\frac{1}{2}e_{y}^T\mathbf{P}^{-1}\widetilde{\mathbf{L}}^{E}e_{y} + \frac{1}{2N_{\min}}e_{x}^T\mathbf{L}^{E}e_{x},
\end{align*}
which concludes the proof.
\end{proof}
\subsection{Proof of Lemma \ref{lem:V_ex_dot}}
\begin{proof}  We start the proof using Lemma \ref{lem:fast_slow_dynamics}. We have, 
\begin{align*}
\frac{dV_{e_{x}}}{dt} = e_{x}^{T}\dot{e}_x =-e_{x}^T\mathbf{W}\mathbf{L}^{I}e_{x} - e_{x}^T\mathbf{W}\mathbf{L}^{E}e_{x} - e_{x}^T\mathbf{W}\mathbf{L}^{E}\mathbf{U}e_{y}.  
\end{align*}
Next, using \eqref{eq:inequality_AMGM} where we choose $\eta =1$, the last term in the above equation can be expressed as 
\begin{align}
        -& e_{x}^T\mathbf{W}\mathbf{L}^{E}\mathbf{U}e_{y}  = -(e_{x}^{T}\mathbf{W}\sqrt{\mathbf{L}^{E}})(\sqrt{\mathbf{L}^{E}}\mathbf{U}e_{y}), \nonumber\\
&\leq \frac{1}{2}\|\sqrt{\mathbf{L}^E}\mathbf{W}e_{x}\|^2 + \frac{1}{2}\|\sqrt{\mathbf{L}^E}\mathbf{U}e_{y}\|^2, \nonumber\\
&=  \frac{1}{2}e_{x}^T\mathbf{W}\mathbf{L}^{E}\mathbf{W}e_{x} + \frac{1}{2}e_{y}^T\widetilde{\mathbf{L}}^{E}e_{y}. \label{eq:inequality_e_x}
\end{align}
Considering the first term in the above equation, using \eqref{Eq:fast_variable_area_v2} we observe that $e_{x}^{T}\mathbf{W} = x^{T}\mathbf{W}^{2} = x^{T}\mathbf{W} = e_{x}^{T}$, since $\mathbf{W}^2 = \mathbf{W}$. This leads to  
\begin{align}\label{eq:inequality_e_x2}
    \frac{1}{2}e_{x}^T\mathbf{W}\mathbf{L}^{E}\mathbf{W}e_{x} \leq \frac{1}{2}e_{x}^T\mathbf{W}\mathbf{L}^{E}e_{x}.
\end{align}
In view of \eqref{eq:inequality_e_x} and \eqref{eq:inequality_e_x2},  we have
\begin{align*}
\frac{dV_{e_{x}}}{dt} \leq -e_{x}^T\mathbf{W}\mathbf{L}^{I}e_{x} - \frac{1}{2}e_{x}^T\mathbf{W}\mathbf{L}^{E}e_{x} +  \frac{1}{2}e_{y}^T\widetilde{\mathbf{L}}^{E}e_{y}.
\end{align*}
This completes the proof. 
\end{proof}

\bibliographystyle{IEEEtran}
\bibliography{IEEEfull,ncs}

\end{document}